\documentclass{amsart}

\usepackage{amsfonts,amssymb,amsmath,amsthm}
\usepackage{url}
\usepackage{enumerate}
\usepackage{comment}

\usepackage{graphics}
\usepackage{graphicx}

\newtheorem{thm}{Theorem}
\newtheorem{lemma}{Lemma}
\newtheorem{definition}{Definition}	

\newtheorem{cor}{Corollary}

\newtheorem{exam}{Example}

\def\E{{\mathcal{ E }}}
\def\F{{\mathcal{ F }}}
\def\B{{\mathcal{ B }}}
\def\I{{\mathcal{ I }}}
\def\A{{\mathcal{ A }}}
\def\C{{\mathcal{ C }}}

\def\M{{\mathcal{ M }}}

\def\P{{\mathcal{ P }}}
\def\K{{\mathcal{ K }}}
\def\k{{\kappa}}
\def\t{{\theta}}
\def\l{{\lambda}}
\def\fin{{\rm{ fin }}}
\def\ult{{\rm{ Ult }}}

\def\aut{{\rm{ Aut }}}
\def\lo{{\rm{ LO }}}
\def\no{{\rm{ NO }}}
\def\na{{\rm{ NA }}}

\title{Ramsey property for Boolean algebras with ideals and $\P(\omega_1)/\fin$}
\author{Dana Barto\v{s}ov\'a}

\address{Department of Mathematics\\
University of Toronto\\
Bahen Center\\
40 St. George St.\\
Toronto\\
Ontario\\
Canada\\
M5S 2E4}
\email{dana.bartosova@utoronto.ca}

\begin{document}

\maketitle

	\textbf{Abstract.} We apply the Dual Ramsey Theorem of Graham and Rothschild to prove the Ramsey property for classes of finite Boolean algebras with distinguished ideals. This allows us to compute the universal minimal flow of the group of automorphisms of $\P(\omega_1)/\fin,$ should it be isomorphic to $\P(\omega)/\fin$ or not. Taking Fra\"iss\'e limits of these classes, we can compute universal minimal flows of groups of homeomorphisms of the Cantor set fixing some closed subsets.

\section{Introduction}

The Boolean algebra $\P(\omega)/\fin$ and its Stone space $\omega^*=\beta\omega\setminus \omega$ are objects in set theory and set-theoretic topology of high interest. It was a big surprise that the question whether $\P(\omega)/\fin$ can be isomorphic to its higher analogue $\P(\omega_1)/\fin$ is not easy to settle (here $\omega_1$ is the order type of the first uncountable cardinal).  This question was first asked in 1970's at a seminar in Katowice and therefore is known as Katowice Problem. It is easily seen to consistently have a negative solution, e.g. under the Continuum Hypothesis. The first reaction of most is that such an isomorphism cannot exist and it has been studied by many great mathematicians. Balcar and Frankiewicz (\cite{BFr}) showed that this problem is really specific to $\omega$ and $\omega_1,$ in particular $\P(\kappa)/\fin$ cannot be isomorphic to $\P(\l)/\fin$ for any  pair of cardinals $\k, \l$ other than $\omega, \omega_1.$  A natural approach is to assume existence of an isomorphism between $\P(\omega)/\fin$ and $\P(\omega_1)/\fin$ and to study consistency of its consequences. See for instance works of Juris Stepr\={a}ns (\cite{JS}) or David Chodounsk\'y (\cite{Ch}). Recently, Klass Pieter Hart proved that another consequence is existence of an automorphism of $\P(\omega)/\fin$ which is not induced by an almost permutation. This was our original motivation to study dynamics of groups of automorphisms of $\P(\omega)/\fin$ and $\P(\omega_1)/\fin.$

A connection between dynamics of groups of automorphisms of structures and finite Ramsey theory was first observed by Pestov in \cite{P2}, who showed that the group $G$ of automorphisms of $(\mathbb{Q},<)$ is extremely amenable, i.e. every action of $G$ on a compact Hausdorff space has a fixed point. The first example of an extremely amenable group was constructed by Herer and Christensen (\cite{HCh}) in the setting of pathological submeasures answering the question of Mitchell whether such groups exist (\cite{M}). Generalizing Pestov's result, Kechris, Pestov and Todor\v{c}evi\'c showed in \cite{KPT} that groups of automorphisms of numerous structures are examples of extremely amenable groups. They developed general framework connecting topological dynamics, Fra\"iss\'e classes and the Ramsey property to give explicit descriptions of universal minimal flows of groups of automorphisms of structures. Their work started a new vital branch of research pursued by for example Lionel Nguyen van Th\'e (\cite{vT}, \cite{vT2}), Miodrag Soki\'{c}({\cite{MS}}), Julien Melleray and Todor Tsankov (\cite{MT}). Moreover, new connections were found, e.g. with ergodic theory (\cite{Ke}, \cite{AKL}).

In \cite{KPT}, the authors identified universal minimal flows with subspaces of spaces of linear orderings. However, as they note at the end of their paper, there is nothing specific to linear orderings in this theory, and other relations serve the same purpose for other classes of structures not considered in the paper. This generalization was carried out by Lionel Nguyen van Th\'e (\cite{vT2}), where the author proves results analogous to those in \cite{KPT} for extensions by  at most countably many new relations. Working with extensions is necessary, because many classes of finite structures are not Ramsey classes, but they have finite Ramsey degrees and thus adding new relations produces a class with the Ramsey property. Therefore, computing Ramsey degrees is essential for dynamical applications, as well as for classification results (see \cite{Fo}, \cite{Fo2}, \cite{MS}, \cite{vT2}).

The paper is organized as follows. In the second section, we introduce basic objects under consideration. In the third section, we remind the reader of the connection between Fra\"iss\'e classes, the Ramsey property and the ordering property given in \cite{KPT}. In the fourth section, we prove the Ramsey property for classes of finite Boolean algebras with ideals. In the last section, we apply the Ramsey property to compute the universal minimal flow of the group of automorphisms of $\P(\omega_1)/\fin,$ quotients of other power set algebras and groups of homeomorphisms of the Cantor set fixing generic closed subsets.
 
\section{Preliminaries}

In this section, we introduce basic notions both from set theory and topological dynamics.

\subsection{Set theory}

\begin{definition}[(Prime) ideal]
Let $X$ be a set and let $\I$ be a collection of subsets of $X$. We say that $\I$ is an ideal on $X$ if it is downwards closed and closed under finite unions:
\begin{itemize}
 \item[(1)] If $A\in \I$ and $B\subset A$ then also $B\in \I,$
 \item[(2)] if $A,B\in \I$ then $A\cup B\in I.$
\end{itemize}
We call $\I$ a prime filter if it moreover satisfies a third condition:
\begin{itemize}
\item[(3)] If  $A, B\subset X$ and   $A\cup B\in\I$ then $A\in \I$ or $B\in \I.$
\end{itemize}

We say that an ideal is \emph{non-trivial}, if $X\notin \I.$ 
\end{definition}

In what follows, we only consider non-trivial ideals.

\begin{exam}\label{E}
The collection of all finite subsets of a set $X$ is an ideal on $X$. If $X$ is a measure space, then the sets of measure $0$ form an ideal. If $X$ is a topological space, then the collection of all meager subsets is an ideal.
\end{exam}

A dual notion to an ideal is the notion of a filter.

\begin{definition}[(Ultra)filter]
A collection $\F$ of subsets of a set $X$ is called a \emph{filter} if it is closed under supersets and finite intersections:
\begin{itemize}
\item[(1)] If $A\in \F$ and $A\subset B$ then also $B\in \F,$ 
\item[(2)] if $A,B\in \F$ then $A\cap B\in\F.$
\end{itemize}
We call $\F$ an \emph{utlrafilter} if it is a maximal filter with respect to inclusion. Equivalently if it moreover satisfies a third condition:
\begin{itemize}
\item[(3)] For every  $A\subset X$ either   $A\in\F$ or   $X\setminus A\in \F.$
\end{itemize}
A filter $\F$ is called \emph{non-trivial} if $\emptyset\notin \F.$
\end{definition}

In this paper, we only consider non-trivial filters.

It is easy to see that if $\I$ is an ideal then $\F=\{X\setminus A:A\in \I\}$ is a filter. We say that $\F$ is a \emph{dual filter} to the ideal $\I.$ 

\begin{exam}
Filters dual to the ideals in Example \ref{E} are the filter of cofinite sets, the filter of sets of positive measure and the filter of subsets of the second category.
\end{exam}

Let us recall the definition of the Boolean algebra $\P(\omega_1)/\fin$ and similar algebras. If $X$ is a set and $\P(X)$ denotes its power set, then $\P(X)$ with the operations of $\cup,\cap,$ complementation, $\emptyset$ and $X$ is a Boolean algebra. If $\l$ is at most the cardinality of $X$ then the set $\I$ of all subsets of $X$ of cardinality less than $\l$ is an ideal on $\P(X).$ The quotient algebra of $P(X)$ by the ideal $\I$ consists of equivalence classes $[E]=\{F\subset X: \Delta(E,F)<\l\}$ for $E\subset X$ and $\Delta(E,F)$ the symmetric difference of $E$ and $F.$ It is easy to see that the operations on $\P(X)$ preserve the equivalence classes.

Below we give an abstract definition of a Boolean algebra, ideal and a quotient algebra.

\begin{definition}[Boolean algebra]
$\B=(B,\vee,\wedge,0,1,\neg)$ is a Boolean algebra if $\vee,\wedge$ are binary operations called \emph{join} and \emph{meet} respectively, $\neg$ is a unary operation called complement and $0,1$ are constants such that for all $a,b,c\in \B$ the following axioms hold:
\begin{itemize}
\item $(a\vee b)\vee c=a\vee (b\vee c)$ and $(a\wedge b) \wedge c=a\wedge(b\wedge c)$ - associativity,
\item $a\vee b=b\vee a$ and $a\wedge b=b\wedge a$ - commutativity,
\item $a\vee 0=a$ and $a\wedge 1=a$ - identity,
\item $a\vee (b\wedge c)=(a\vee b)\wedge (a\vee c)$ and $a\wedge (b\vee c)=(a\wedge b)\vee (a\wedge  c)$ - distributivity,
\item $a\vee \neg a=1$ and $a\wedge \neg a=0$ - complements.
\end{itemize}
\end{definition}

There is a canonical partial ordering $<$ on a Boolean algebra $\B,$ given by $a<b$ if and only if $a\wedge b=a $ ( if and only if $a\vee b=b$.)

An \emph{atom} of a Boolean algebra $\B$ is a non-zero element $a\in \B$ such that $0$ is the only element $<$-below $a.$ A Boolean algebra is called \emph{atomless} if it does not have any atoms. Every finite Boolean algebra has atoms. There is exactly one countable atomless Boolean algebra up to an isomorphism - the algebra of all clopen subsets of the Cantor set.

An ideal on a Boolean algebra is defined similarly as an ideal on a set. An \emph{ideal} on a Boolean algebra $\B$ is a subset of $\B$ closed under finite meets and downwards closed under the canonical partial order.

If $\B$ is a Boolean algebra and $\I\subset \B$ an ideal on $\B,$ then $\B/\I$ denotes the \emph{quotient algebra} of $\B$ by $\I$ consisting of equivalence classes 
$$
[a]=\{b\in \B: (a\wedge \neg b)\vee (\neg a\wedge b)\in \I\}
$$
for $a\in \B.$ The operations on $\B/\I$ are inherited from $\B.$ Note that $(a\wedge \neg b)\vee (\neg a\wedge b)$ corresponds to the symmetric difference if $\B$ is a power set algebra.

The category of Boolean algebras with homomorphisms is equivalent to the category of compact Hausdorff zero-dimensional spaces via the Stone duality.

\subsection*{Stone representation theorem}\label{Stone}
In 1936, M.H. Stone \cite{MHS} proved that every Boolean algebra $\mathcal{B}$ is isomorphic to the Boolean algebra of all clopen subsets of a compact totally disconnected Hausdorff space $\ult(\mathcal{B})$. The points of $\ult(\mathcal{B})$ are all ultrafilters on $\mathcal{B}$ and the sets $A^*=\{u\in \ult(\mathcal{B}): A\in u\}$ for $A\in \mathcal{B}$ form a clopen base of the topology on $\ult(\mathcal{B}).$ This gives a one-to-one correspondence that also extends to homomorphisms: If $f:\mathcal{B}\to\mathcal{C}$ is a homomorphism between two Boolean algebras, then $\ult(f):\ult(\mathcal{C})\to \ult(\mathcal{B})$ given by $$u\mapsto \mbox{``the  ultrafilter on } \mathcal{B} \mbox{ generated by } \bigcup\{f^{-1}(A):A\in u\}\mbox{''}$$ is a continuous map. If $f$ is injective, then $\ult(f)$ is surjective. 
In terms of category theory, $\ult$ is a contravariant functor giving an equivalence between the category of Boolean algebras with homomorphisms and the category of compact totally disconnected Hausdorff spaces with continuous mappings.

\begin{exam}

The \v{C}ech-Stone compactification $\beta \omega$ of discrete $\omega$ is the Stone space of $\P(\omega).$ As we noted in the Introduction, $\omega^*=\beta\omega\setminus\omega$ is the Stone representation of the Boolean algebra $\P(\omega)/\fin.$ The Stone space of the countable atomless Boolean algebra is the Cantor space.
\end{exam}

\subsection{Topological dynamics}

The central notion of topological dynamics is a continuous action $\pi:G\times X\to X$ of a topological group $G$ on a compact Hausdorff space $X.$ We call $X$ a \emph{$G$-flow} and omit $\pi$ if the action is understood and write $gx$ instead of $\pi(g,x).$ A \emph{homomorphism} of $G$-flows $X$ and $Y$ is a continuous map $\phi:X\to Y$ respecting the actions of $G$ on $X$ and $Y$, i.e. $\phi(gx)=g\phi(x)$ for every $g\in G,$ $x\in X$ and $y\in Y.$ We say that $Y$ is a \emph{factor} of $X,$ if there is a homomorphism from $X$ onto $Y.$ Every $G$-flow has a minimal subflow, a minimal closed subspace of $X$ invariant under the action of $G.$ Among all minimal $G$-flows, there is a maximal one - the \emph{universal minimal flow} $M(G).$ It means that every other minimal $G$-flow is a factor of $M(G).$ For introduction to topological dynamics see \cite{dV}.

\subsection*{Automorphism groups}
Let $\B$ be a Boolean algebra and let $G$ be the group of automorphisms of $\B$. In what follows, we consider $G$ as a topological group with the topology of pointwise convergence. The topology is given by a basis of neighbourhoods of the neutral element consisting of open subgroups 
$$G_A=\{g\in G:g(a)=a, a\in A\}
$$
 where $A$ is a finite subalgebra of $\B.$ We can observe that a subset  $H$ of $G$ is closed if and only if it contains every $g\in G$ such that for any finite subalgebra $A$ of $\B$ there exists an $h\in H$ with $g|A=h|A.$

If every partial isomorphism between two finite subalgebras of $\B$ can be extended to an automorphism of the whole algebra $\B,$ then we say that $\B$ is \emph{homogeneous}. Equivalently, a Boolean algebra $\mathcal{B}$ is homogeneous if for every $b\in\mathcal{B}$ the relative Boolean algebra $\mathcal{B}|b=\{c\in\mathcal{B}:c\leq b\}$ is isomorphic to $\mathcal{B}.$ 

\subsection*{Linear orderings}

We will consider the space $\lo(\B)$ of all linear orderings of a Boolean algebra as a subspace of the compact space $2^{\B\times\B}$ with the product topology. It is easy to see that $\lo(\B)$ is closed in $2^{\B\times\B}$ and therefore compact. The topology on $\lo(\B)$ is generated by basic open sets of the form
$$
(A,\prec)=\{<\in\lo(\B): \ <|A=\prec\},
$$
where $A$ is a finite subalgebra of $\B$ and $\prec$ is a linear ordering on $\B.$

If $G$ is the group of automorphisms of $\B,$ we let $G$ act on $\lo(\B)$ by
$$
a (g<) b \rm{\ if \ and \ only \ if \ } g^{-1}a<g^{-1}b,
$$
for any $g\in G,$ $<\in\lo(\B)$ and $a,b\in \B.$

\section{Fra\"iss\'e classes, the Ramsey and the ordering properties}

In this section, we introduce the main ingredients of the theory developed by Kechris, Pestov and Todor\v{c}evi\'c in \cite{KPT}.

The first ingredient are  Fra\"iss\'e classes. The original definition includes classes of finitely-generated structures. For our purpose it is however enough to introduce the definition for classes of finite structures.

\begin{definition}[Fra\"iss\'e class]
A class of finite structures $\mathcal{K}$ of a given language is called a \emph{Fra\"iss\'e class} if it is countable, contains structures of arbitrary finite cardinality and it satisfies the following conditions:
\begin{itemize}
\item[(HD)] Hereditary property:  if $A$ is a finite substructure of $B$ and $B\in\mathcal{K},$ then also $A\in \mathcal{K}.$
\item[(JEP)] Joint embedding property: if $A,B\in \mathcal{K}$ then there exists a $C\in\mathcal{K}$ in which both $A$ and $B$ embed.
\item[(AP)] Amalgamation property: if $A,B,C\in\mathcal{K}$ and $i:A\to B$ and $j:A\to C$ are embeddings, then there exist $D\in\mathcal{K}$ and embeddings $k:B\to D$ and $l:C\to D$ such that $k\circ i=l\circ j.$
\end{itemize}
\end{definition}

If $\A$ is a countable locally-finite (finitely-generated substructures are finite) homogeneous structure then the class of all finite substructures of $\A$ up to an isomorphism form a Fra\"iss\'e class. Conversely, if $\K$ is a Fra\"iss\'e class then there is a unique countable locally-finite homogeneous structure $\A$ (up to an isomorphism) such that $\K$ is the class of all finite substructures of $A$ up to an isomorphism. We call $\A$ a \emph{ Fra\"iss\'e  limit} of $\K.$ This correspondence was first described by Fra\"iss\'e in \cite{F}.

Fra\"iss\'e  limits of Fra\"iss\'e  classes are exactly countable locally-finite homogeneous structures. They are unique up to an isomorphism by $\omega$-categoricity. 

\begin{exam}
The countable atomless Boolean algebra, the random graph, the random hypergraph, $(\mathbb{Q},<)$ are Fra\"iss\'e  limits of the class of finite Boolean algebras, finite graphs, finite hypergraphs, finite linear orderings respectively.
\end{exam}

\medskip

In 1928, Ramsey proved the following combinatorial property to classify binary relations. It is now known as Ramsey's theorem and it was a building stone for nowadays fast developing Ramsey theory.

\begin{thm}[Ramsey's Theorem]
For every $k<n\in \mathbb{N}$ and $l\in\mathbb{N},$ there exists $N\in\mathbb{N}$ such that for every colouring of $k$-element subset of $N$ by $l$-many colours there is an $n$-element subset $X$ of $N$ such that all $k$-element subsets of $X$ have the same colour.
\end{thm}

In the 1970's and 1980's the Ramsey property for structures was first studied.

\begin{definition}[Ramsey property]
 A class $\mathcal{K}$ of finite structures satisfies the \emph{Ramsey property} if for every $A\leq B\in\mathcal{K}$ and  a natural number $k\geq 2$ there exists $C\in\mathcal{K}$ such that 
 $$
 C\to (B)^A_k,
 $$
 i.e. for every colouring of copies of $A$ in $C$ by $k$ colours, there is a copy $B'$ of $B$ in $C,$ such that all copies of $A$ in $B'$ have the same colour $c_{B'}$. We then say that $B'$ is \emph{monochromatic} in the colour $c_{B'}.$
\end{definition}

The classes of finite linearly ordered ($K_n$-free) graphs satisfy the Ramsey property by \cite{NR1} and \cite{NR2}, the class of finite posets with linear orderings extending the partial order satisfy the Ramsey property by \cite{MS}. 

In the last section, we prove the Ramsey property for classes of finite Boolean algebras with unary predicates represented as ideals. We rely on the corresponding result for the class of finite Boolean algebras, today known as the Dual Ramsey Theorem.

\begin{thm}[Dual Ramsey Theorem]\cite{GR}\label{GR}
The class of finite Boolean algebras satisfies the Ramsey property.
\end{thm}

In \cite{KPT}, the following property was found to be equivalent for  Fra\"iss\'e classes  to a space of linear orderings being a minimal flow for the group of automorphisms of their Fra\"iss\'e limits.

\begin{definition}[Ordering property]
Let $L\supset \{<\}$ be a signature and let $L_0=L\setminus\{<\}.$
Let $\K$ be a class in $L$ and let $\K_0$ denote the reduct of $\K$ to L. We say that $\K$ satisfies the \emph{ordering property}, if for every $A\in\K_0,$ there is $B\in K_0$ such that whenever $\prec$ is a linear ordering on $A,$ $\prec'$ is a linear ordering on $B$ and $(A,\prec),(B,\prec')\in\mathcal{K},$ then $(A,\prec)$ is a substructure of $(B,\prec').$ 
\end{definition}

Let $\K$ and $\K_0$ be as above. We call $\K$ \emph{order forgetful} if whenever $(A,<^A),(B,<^B)\in \K$ then $A\cong B$ if and only if $(A,<^A)\cong(B,<^B)$ 

It is easy to see that if $\K$ is order forgetful, then it satisfies the ordering property. Moreover, Proposition 5.6 in \cite{KPT} shows that $\K$ satisfies the Ramsey property if and only if $\K_0$ satisfies the Ramsey property.

The main results of Kechris, Pestov and Todor\v{c}evi\'c in \cite{KPT} are theorems 4.8 and 7.5. We combine excerpts of these theorems relevant to this paper in the following theorem.

\begin{thm}[\cite{KPT}]\label{T2}
Let $L\supset\{<\}$ be a signature, $L_0=L\setminus\{<\},$ $\K$ a reasonable Fra\"iss\'e class in $L$ with $<$ interpreted as a linear order,  and let $\K_0$ be the reduct of $\K$ to the language $L_0.$ Let $\mathbf{F}$ denote the Fra\"iss\'e limit of $\K,$ $\mathbf{F}_0$ the reduct of $\mathbf{F}$ to $L,$ and let $G, G_0$ be the automorphism groups of $\mathbf{F}, \mathbf{F}_0$ respectively. Let $X_{\K}=\overline{G<^{\mathbf{F}}}.$

If $\K$ satisfies the Ramsey property and the ordering property, then $G$ is extremely amenable and $X_{\K}$ is the universal minimal flow of $G_0.$
\end{thm}

\section{Classes of finite Boolean algebras with ideals}

In this section, we  consider Fra\"iss\'e  classes of finite linearly ordered Boolean algebras with an increasing chain of  ideals and show that they satisfy the Ramsey and the ordering properties. Since the Fra\"iss\'e limit of the class of Boolean algebras is the countable atomless Boolean algebra whose Stone space is the Cantor set, this allows us to compute universal minimal flows of groups of homeomorphisms of the Cantor set fixing some closed sets (given by the generic ideals).

We also consider classes of isomorphism types of finite subalgebras of $\P(\k)/[\k]^{<\l}$ for any pair of cardinals $\l\leq \k.$ These cannot be countable and therefore not Fra\"iss\'e classes, however they satisfy all other requirements in the definition of the Fra\"iss\'e class and they admit an extension that satisfies the Ramsey and the ordering properties. This allows us to compute the universal minimal flow of groups of automorphisms of  $\P(\omega_1)/\fin.$

Let $L=\{\vee,\wedge,0,1,\neg\}$ be the language of Boolean algebras. Let $J$ be a linearly ordered set. We denote by $L_J$  the language $L$ expanded by $|J|$-many unary symbols $\left<P_{j}:j\in J\right>$, i.e. $L_{J}=\{\vee,\wedge,0,1,\neg, P_{j}:j\in J\}.$ To simplify the notation, if $A$ is a structures in the language $L_J$ and we often write $P_j$ in place of $P_j^A$ for every $j\in J.$

We will consider three classes of finite Boolean algebras in the language $L_{J}.$ 

\begin{definition}\label{classes}

\begin{itemize}

\noindent
\item[(1)] Let $\B_{J}$ denote the class of isomorphism types of finite Boolean algebras in the language $L_{J},$ where each $P_{j}$ for $j\in J$ is interpreted as an ideal and   $P_{i}\subset P_{j}$ for $i<j\in J.$ Moreover, to ensure the amalgamation property, we require that  every $A\in \B_J$ has at least one atom not in any $P_j$ for $  j\in J.$ 

\item[(2)] Let $\B_u$ denote the class of isomorphism types of finite Boolean algebras in the language $L_{1}$ with $P_0$ interpreted as an ideal such that every $A\in \B_J$ has exactly one atom not in $P_0.$ For 
every $A\in \B_u$ all atoms but one are in $P_0.$

\item[(3)] Let $\B_J^u$ denote the class of isomorphism types of finite Boolean algebras in the language $L_{J}$ where each $P_{j}$ for $j\in J$ is interpreted as an ideal and   $P_{i}\subset P_{j}$ for $i<j\in J.$ Moreover, for every $A\in \B_J^u$ all atoms but one are in one of $P_{j}.$ 

\end{itemize}
\end{definition}

If $J$ is countable then all the classes in Definition \ref{classes} are countable and it is easy to see that they are Fra\"iss\'e classes. If $J$ is uncountable then $\B_{J}$ and  $\B_J^u$  are uncountable as well and therefore not Fra\"iss\'e classes. However, they satisfy all the other conditions in the definition of a Fra\"iss\'e class. 

Therefore, if $J$ is countable, we can consider Fra\"iss\'e limits of all the classes in Definition \ref{classes}. Their limits will be countable atomless Boolean algebras with each $P_j$ for $j\in J$ interpreted as an ideal. We call inteprations of the $P_{j}$'s for $j\in J$ in the limit  \emph{generic ideals}. Since the Stone space of the countable atomless Boolean algebras is the Cantor set, dual filters to generic ideals correspond to generic closed subsets of the Cantor set. In the case of our classes, they will correspond to a decreasing chain of closed subsets, a point and a decreasing chain of closed subsets with exactly one point in the intersection.

We will show that we can equip structures in each class in Definition \ref{classes} with linear orderings and obtain a class of order forgetful structures satisfying the conditions in the definition of a Fra\"iss\'e class. We alter the definition of natural orderings in \cite{KPT} respecting the enumeration of the ideals. 

We first specify which orderings of atoms we allow. Let $A$ be a finite Boolean algebra in the language $L_J.$ We call an ordering $<$ of atoms of $A$ \emph{proper} if for every  two atoms $a,b\in A$ and $i<j$  if $a\in P_i$ and $b\in P_j\setminus P_i$ or $b\notin P_k$ for any $k\in J,$ then $a<b.$

\begin{definition}
Let $\B$ be one of $\B_J,\B_u, \B_J^u$ and let $A\in \B.$ We say that a linear ordering $<$ on $A$  is \emph{natural} if it is an antilexicographical extension of a proper ordering of atoms of $A.$  We denote by $\na(\B)$ the class of all naturally ordered algebras from $\B.$
\end{definition}

The class $\na(\B)$ is obviously order forgetful for any choice of $\B$, therefore it trivially satisfies the ordering property. Indeed, if $A\in \B_J$ and $<,<'$ are two natural orderings on $A$ then $(A,<)\cong (A,<').$ In other words, we can take $A$ in place of $B$ in the definition of the ordering property.

In \cite{KPT}, the authors proved that naturally ordered finite Boolean algebras (without predicates) form a Fra\"iss\'e class. Working along their proof, we show that the same is true for the class $\na(\B)$ (except for the condition of countability if $J$ is uncountable and $\B=\B_J$ or $\B=\B_J^u$).

\begin{lemma}
Let $\B$ be one  $\B_J, \B_u, \B_J^u$ and let $L_{\B}$ be the language of $\B.$  Then
$\na(\B)$ satisfies (HP), (JEP) and (AP) properties in the definition of a Fra\"iss\'e class.
\end{lemma}

\begin{proof}
We first show that $\na(\B)$ is hereditary. Let $B\in \B$ and let $A$ be its subalgebra. Suppose that $<$ is a natural ordering of $B$ given by an ordering of its atoms $b_1<b_2<\ldots<b_n.$ Let $a_1<a_2<\ldots<a_k$ be the atoms of $A.$ Write $a_i=b_{i1}\vee b_{i2}\vee\ldots\vee b_{ij_i},$ where $b_{i1},b_{i2},\ldots,b_{ij_i}$ are atoms of $B$ in $<$-increasing order. Then $a_i\in P_h$ for a $P_h\in L_{\B}$ if and only if $b_{ij_i}\in P_h.$ This holds because the ideals are ordered by inclusion  and $<$ respects this order while listing elements not in any $P_h$ as the last. Moreover, if $\B$ is either $\B_u$ or $\B_J^u,$ then $A$ has exactly one atom not in any $P_h$ since $B$ does.  Therefore $A\in \B.$

It remains to prove (AP), since (JEP) follows by amalgamation along the two element algebra. Let $(A,<_A),(B,<_B),(C,<_C)\in\na(\B)$ with atoms $a_1<_Aa_2<_A\ldots<_Aa_k,$ $b_1<_Bb_2<_B\ldots<_Bb_n,$ $c_1<_C c_2<_C\ldots<_Cc_m$ respectively. Suppose that there are embeddings $f: (A,<_A)\to (B,<_B)$ and $g:(A,<_A)\to (C,<_C).$ We look for $(D,<_D)$ and embeddings $r:(B,<_B)\to (D,<_D)$ and $s:(C,<_C) \to (D,<_D)$ such that $r\circ f=s\circ g.$

Write $f(a_i)=b_{i1}\vee b_{i2}\vee\ldots\vee b_{ij_i}$ with $b_{i1},b_{i2},\ldots,b_{ij_i}$ some of the atoms of $B$ in the increasing order and $g(a_i)=c_{i1}\vee c_{i2}\vee \ldots \vee c_{il_i}$ for $c_{i1},c_{i2},\ldots,c_{il_i}$ some of the atoms of $C$ in the increasing order. Notice that $a_i\in P_h$ if and only if $b_{ij_i},c_{il_i}\in P_h$ by the same argument as for (HP).

We let the atoms of $D$ to be $\{\overline{b_{ij}}\}_{1\leq i\leq k,1\leq j\leq j_i}$ and $\{\overline{c_{ij}}\}{1\leq i \leq k,1\leq j\leq l_i }$ all disjoint except for $\overline{b_{ij_i}}=\overline{c_{il_i}}$ for $1\leq i\leq k.$ We need to determine how  $<_D$ will behave on these atoms and where the atoms of $B,C$ will be send by $r,s$ respectively.

First consider $\{\overline{b_{1j}}\}_{1\leq j\leq j_1}\cup \{\overline{c_{1j}}\}_{1\leq j\leq l_1}$ and set 
$$
\overline{b_{11}}<_D \overline{b_{12}} <_D\ldots <_D \overline{b_{1j_i}}, \ \overline{c_{11}}<_D\overline{c_{12}} <_D \ldots <_D \overline{c_{1l_1}}=\overline{b_{1j_1}}
$$
and define $<_D$ to be arbitrary on the rest to be proper. 

For any linear ordering $<$, we denote $(-\infty,a]=\{x:x\leq a\}$ and $(a,b]=\{x:a<x\leq v=b\}.$ Then we let
$$
r(b_{11})=\bigvee (-\infty,\overline{b_{11}}], \  r(b_{12})=\bigvee (\overline{b_{11}},\overline{b_{12}}], \ldots, \ r(b_{1j_1})=\bigvee (\overline{b_{1j_1-1}},\overline{b_{1j_1}}] 
$$
$$
s(c_{11})=\bigvee (-\infty,\overline{c_{11}}],\ s(c_{12})=\bigvee (\overline{c_{11}},\overline{c_{12}}],\ldots, \ s(c_{1l_1})=\bigvee (\overline{c_{1l_1-1}},\overline{c_{1l_1}}]. 
$$

Then $b_{1t}\in P_h$ if and only if $r(b_{1t})\in P_h$  for every $t=1,\ldots,j_1$ and $h\in J,$ since $<$ on $\{\overline{b_{1j}}\}_{1\leq j\leq j_1}\cup \{\overline{c_{1j}}\}_{1\leq j\leq l_1}$ is proper. Similarly, $c_{1t}\in P_h$ if and only if $s(c_{1t})\in P_h$ for every $t=1,\ldots,l_1$ and $h\in J.$

Now we extend $<_D$ to  $\{\overline{b_{1j}}\}_{1\leq j\leq j_1}\cup \{\overline{c_{1j}}\}_{1\leq j\leq l_1} \cup \{\overline{b_{2j}}\}_{1\leq j\leq j_2}\cup \{\overline{c_2j}\}_{1\leq j\leq l_2}$ and require that the maps $b_{ij}\mapsto \overline{b_{ij}}$ and $c_{ij}\mapsto \overline{c_{ij}}$ are order preserving and extend otherwise arbitrarilyto be proper. We extend $r$ and $s$ in the following way:
\begin{eqnarray*}
r(b_{21})=\bigvee (-\infty,\overline{b_{21}}],\ldots,\ r(b_{2j_2})=\bigvee (\overline{b_{2j_2-1}},\overline{b_{2j_2}}] \\
s(c_{21})=\bigvee (-\infty,\overline{c_{21}}],\ldots, \ s(c_{2l_2})=\bigvee (\overline{c_{2l_2-1}},\overline{c_{2l_2}}].
\end{eqnarray*} 

We proceed in the same manner to define $<_D$ on all atoms of $D$ and $r,s$ on all atoms of $A,B$ respectively.
\end{proof}

The last remaining ingredient we need to show is that all the classes in Definition \ref{classes} satisfy the Ramsey property. We first introduce some notation.

Let $\na(\B)$ be one of $\na(B_J), \na(B_u), \na(B_J^u).$ Let $(X,<_X),(Y,<_Y)\in \na(\B)$ with atoms $x_1,x_2,\ldots, x_n$ and $y_1,y_2,\ldots, y_m$ respectively. Denote by
$$A=X \ast Y$$ 
the algebra with  atoms $x_1,\ldots,x_n, y_1,\ldots, y_m.$ The natural order $<_A$ on $A$ is given by $x_i<_A y_j$ for every $(i,j)\in\{1,2,\ldots,n\}\times\{1,2,\ldots,m\}$ and $x_i<_A x_j$ if and only if $x_i<_X x_j$ and $y_i<_Ay_j$ if and only if $y_i<_Y y_j.$ The interpretation of $P_j$'s in $A$ are given by their interpretation in $X$ and $Y$: $x_i\in P_j^A$ if and only if $x_i\in P_j^X$ and $y_i\in P_j^A$ if and only if $y_i\in P_j^Y.$ We will only perform this operation when it provides a well defined algbebra, it means belonging to the naturally ordered class under consideration.

Let $(A,<)\in \na(\B)$  and let $j_0$ be the $J$-minimal element such that $P_{j_0}^A\neq \emptyset.$ Denote by $A_0$ the algebra generated by the atoms of $A$ that lie in $P_{j_0}^A$  and denote by $A_1$ the algebra generated by the remaining atoms. Suppose that $X$ is a subalgebra of $A_0$ generated by $n$ atoms $x_1< x_2<\ldots < x_n$ and $Y$ is a subalgebra of $A_1$ generated by $m$ atoms $y_1 <y_2 <\ldots < y_m$ and $n>m.$   
We define
$$X\circ Y$$ 
to be the subalgebra of $A$ with atoms $b_1, b_2,\ldots, b_n$ where $b_i=x_i$ if $i=1,2,\ldots,n-m$ and $b_i=x_i\vee y_{i-(n-m)}$ if $i>n-m.$

If $X$ is a naturally ordered finite Boolean algebra without a predicate and $j\in J,$ define $$X^{j}$$ to be the algebra $X$ with all of its elements in the predicate $P_{j}.$ Notice that $X^j$ is not a structure in $\na(\B)$ for any choice of $\B.$
 
If $X\in \na(\B),$ then $X_r$ denotes the reduct of $X$ to the language of naturally ordered Boolean algebras (without predicates). 

\begin{thm}\label{RP1}
$\na(\B_J)$ satisfies the Ramsey property.
\end{thm}

\begin{proof}
Let $A\leq B$ be two structures in $\na(\B_{J})$. Denote by  $M=\{j\in J:B\cap \P_{j}\neq\emptyset\}$ and let $j_0$ be the $J$-minimal element of $M.$ Notice that if $A\cap P_{j}\neq \emptyset$ then $j\in M.$ We proceed by induction on $n=|M|.$ 

If $n=0,$ then it is just the Dual Ramsey Theorem. 
 
Suppose that the statement is true for all $M\subset J$ of size $n.$ We will prove the statement for $M$ of size $n+1.$ Let $A_0$ denote the algebra generated by atoms of $A$ in $P_{j_0}$ and $A_1$ the algebra generated by the remaining atoms of $A.$ Similarly, we define $B_0$ and $B_1.$ Let $C_0$ be an algebra given by the Dual Ramsey Theorem  such that
$$
C_0\to (B_r)^{A_r}_2. 
$$ 
Let $C_1$ be given by the induction hypothesis such that
$$
C_1\to (B_1)^{A_1}_{|  {C_0\choose B}  | \times 2}.
$$

  Let $C=C_0^{\mu}\ast C_1.$ We claim that 
$$
C  \to (B)^A_2.
$$

Let $c:{C\choose B}\to \{0,1\}$ be an arbitrary colouring. For every $D\in {C_1\choose A_1}$ define a colouring 
$$c_D:{C_0\choose A_r}\to \{0,1\} \rm{\  by\ } c_D(E)=c(E^{j_0} \circ D).$$
 Using the Dual Ramsey Theorem, find a monochromatic $B_D\in {C_0\choose B_r}$ in colour $k_D.$ Now consider the colouring 
$$c':{C_1\choose A_1}\to  {C_0\choose B}   \times \{0,1\} \rm{\ given by\ } c'(D)=(B_D,k_D).$$
 By the induction hypothesis, there is a monochromatic $B'_1\in{C_1\choose B_1}$ in colour $(B_0',k_0)$ for some $B_0'\in{C_0\choose B_r}$ and $k_0$ equal to $0$ or $1.$ Let $B'=(B_0')^{j_0}\circ B_1'.$ We verify that ${B'\choose A}$ is monochromatic in colour $k_0.$ Let $A'\in{B'\choose A}.$ 
Then $A'=E^{j_0}\circ D$ for some $E\in {B'_0\choose A_r}$ and $D\in {B'_1\choose A_1}.$ It means that $c'(D)=(B_0',k_0)$ and therefore $c(A')=c_D(E)=k_0$ and we are done.
\end{proof}

\begin{thm}\label{RP2}
$\na(\B_u)$ satisfies the Ramsey property.
\end{thm}

\begin{proof}
Let $A,B\in \na(\B_u).$ It means that $A,B$ are structures in the language $L_1$ and all but one atom of both $A$ and $B$ belong to $P_0.$ Let $C_0$ be the Boolean algebra given by the Dual Ramsey Theorem for $A_r, B_r$ and $2$ colours. Let $C_1$ be the Boolean algebra with exactly one atom. We claim that $C=C_0^{j_0}\ast C_1$ satisfies 
$$
C \to (B)^A_2.
$$
Let $c:{C \choose A}\to \{0,1\}$ be an arbitrary colouring. Then $c$ induces a colouring $c':{C_0\choose A_r}\to \{0,1\}$ by $c'(E)=c(E^{j_0}\circ C_1).$ Let $B_0$ be a $c'$-monochromatic copy of $B_r$ in $C_0.$ Then $B'=B_0^{j_0}\circ C_1$ is a $c$-monochromatic copy of $B$ in $C.$ 
\end{proof}

\begin{thm}
$\na(\B_J^u)$ satisfies the Ramsey property.
\end{thm}

\begin{proof}
Let $A,B\in \na(\B_J^u).$ It means that $A,B$ are structures in the language $L_J$ and all but one atom of both $A$ and $B$ belong to some of the ideals $\left<P_j, j\in J\right>.$
We proceed in the same manner as in the proof of Theorem \ref{RP1} showing that we can always pick $C$ with exactly one atom not in any $P_j$ for $j\in J.$
Let $M=\{j\in J:B\cap \P_{j}\neq\emptyset\}$ and let $j_0$ be the $J$-minimal element of $M.$ We proceed by induction on $n=|M|.$

If $n=0,$ then both $A$ and $B$ are the trivial algebras, so we can take $C$ to be a trivial algebra as well to satisfy the Ramsey property.

If $n=1,$ then this is exactly Theorem \ref{RP2}.

Suppose that the theorem is true for all $M\subset J$ of size  $n$ and we shall prove the statement for $M$ of size $n+1$. This directly follows from the proof of Theorem \ref{RP1}, since we obtain $C_1$ from the induction hypothesis for a pair of algebras with at most $n$-many $P_j$'s non-empty, $C_1$ has exactly one atom not in any $P_j$ for $j\in J$ and therefore $C_1\in \B_J^u.$ Therefore also $C=C_0^{j_0}\circ C_1$ belongs to $\na(\B_J^u)$ and we are done.
\end{proof}

Now we are ready to apply the Theorem \ref{T2} of Kechris, Pestov and Todor\v{c}evi\'c and the following generalization to uncountable structures.

\begin{thm}[\cite{DB}]\label{T}
Let $\mathcal{A}$ be a locally-finite homogeneous structure.  Suppose that $\mathcal{K}$ is a Fra\"\i ss\'e order class that is an order forgetful expansion of the class of finite substructures of $\A.$
Then $\no_{\mathcal{K}}(\mathcal{A})$ is the universal minimal flow for every dense subgroup of $\aut(\mathcal{A}).$
\end{thm}

\section{Applications}

\subsection*{The Cantor set}

Assuming $J$ to be a countable linearly ordered set, we compute universal minimal flows of Fra\"iss\'e  limits of $\B_J, \B_u, \B_J^u$ and their naturally ordered analogues. Since the  Fra\"iss\'e limit $\C$ of the class of finite Boolean algebras is the countable atomless Boolean algebra, Fra\"iss\'e limits of  $\B_J, \B_u,  \B_J^u$ are countable atomless Boolean algebras with one generic ideal for each unary predicate added to the language of Boolean algebras. If $\B$ is one of $\B_J, \B_u,  \B_J^u,$ then the Fra\"iss\'e  limit of $\na(\B)$ is the same as the Fra\"iss\'e limit of $\B$ with a distinguished generic normal ordering.

In \cite{KPT}, it was shown that the universal minimal flow of the group of automorphisms of the countable atomless Boolean algebra $\C$ is the space $\no(\C)$ of normal orderings on $\C$ induced by the natural orderings of its finite subalgebras. Applying Theorem \ref{T2}, we obtain similar results for classes of Boolean algebras considered in this paper.

\begin{thm}
Let $J$ be a countable linear order and let $\B$ be on of the classes $\B_J, \B_u,  \B_J^u.$ Let $\C_{\B}$ and $\C_{\B}^<$ be the Fra\"iss\'e limits of $\B$ and $\na(\B)$ respectively. Then
\begin{itemize}

\item[(a)] the universal minimal flow of the group of automorphisms of $\C_{\B}$ is $\no_{\B}(\C_{\B}),$

\item[(b)] the group of automorphisms of $\C_{\B}^<$ is extremely amenable.
\end{itemize}

\end{thm}

We know that the Stone space $\E$ of the countable atomless Boolean algebra $\C$ is the Cantor set. The universal minimal flow of the group of homeomorphisms of the Cantor set was first computed by Glasner and Weiss in \cite{GW}. They identified the universal minimal flow with the space of maximal chains of closed subsets introduced by Uspenskij in \cite{	U1}:  Let $X$ be a compact space and denote by $\exp X$ the space of closed subsets of $X$ equipped with the Vietoris topology. Then the space $\Phi(X)$ of all maximal chains of closed subsets of $X$ is a closed subspace of $\exp \exp X.$  

Since the universal minimal flow is unique, the space of maximal chains $\Phi(\E)$ and the space of natural orderings $\no(\C)$ must be isomorphic. An explicit isomorphism was given in \cite{KPT}. We cite the theorem with its proof.

\begin{thm}[\cite{KPT}]\label{T3}
Let $G$ be the group of homeomorphisms of the Cantor set $\E.$ There exists an (explicit) $G$-isomorphism $\phi:\Phi(\E)\to \no(\C).$
\end{thm}

\begin{proof}
Given a maximal chain $\M$ of closed subsets of $\E,$ for every clopen subset $C$ of $\E,$ let 
$$
M_A=\bigcap \{M\in \M:M\cap C\neq\emptyset\}.
$$
By the maximality of $\M,$ $M_A\cap C$ is a single point for every clopen set $C.$ If $C\subset D$ are two clopen subsets of $\E,$ then $M_C\supset M_D,$ though they can also be equal. However, if $C$ and $D$ are disjoint, then $M_C$  and $M_D$ are different. If $A$ is a finite subalgebra of $\C$ and $a,b$ are two atoms of $A$ then set $a<_{\M}b$ if and only if $M_a\supset M_b.$ This gives a total ordering  of atoms  of $A$ which in turn induces an antilexicographical ordering on $A.$ These orderings cohere and produce a total order $<_{\M}$ of $\C.$ Then $\phi(\M)=<_{\M}$ is the sought for isomorphism from $\Phi(\E)$ to $\no(\C).$
\end{proof}

If $\I$ is an ideal on $\C$ and $\F$ is the dual filter to $\I$ (i.e. $\F=\{\neg a:a\in \I\}$), then $\{e\in E:\F\subset e\}$ is a closed subset of $\E.$ Therefore if $\{P_j^{C_{\B}}:j\in J\}$ are the generic ideal on $C_\B$ for $\B=\B_J,$ then each filter $\F_j$ dual to $\P_j^{C_{\B}}$ corresponds to a closed subset $E_j$ of $\E.$ If $i<j\in J,$ then $P_i^{C_{\B}}\subset P_j^{C_{\B}},$ so  also $\F_i\subset \F_j$ and consequently $E_i\supset E_j$ (larger filter determines a smaller closed subset). So we have that the increasing chain $\{P_j^{C_{\B}}:j\in J\}$ of generic ideals on $\C_{\B}$ corresponds to a decreasing chain $\{E_j:j\in J\}$ of closed subsets of the Cantor set $\E.$

\begin{thm}\label{T4}
Let $\{E_j:j\in J\}$ be a decreasing chain of closed subsets of $\E$ as above Then the group of homeomorphisms of the Cantor set fixing $E_j$ for every $j\in J$ is the space of chains of maximal closed subsets of $\E$ containing $\{E_j:j\in J\}.$
\end{thm}

\begin{proof}
Let $\phi:\Phi(\E)\to \no_{\B_J}(\C_{\B_J})$ be the isomorphism given in Theorem \ref{T3}.
Let $A$ be a finite substructure of $\C_{\B_J}.$ We first show that if $\M$ is a maximal chain of closed subsets of $\E$ that extends $\{E_j:j\in J\},$ then $\phi(\M)=<_{\M}$ is a normal ordering on $\C_{\B_J}.$ Let $A$ be a finite subalgebra of $\C_{\B_J}.$ Let $a,b$ be atoms of $A$ and suppose that $a\in P_i$ and $b\in P_j\setminus P_i$ for some $i<j.$ Thinking of $a$ and $b$ as clopen subsets of $\E,$ we have that $E_i\subset \neg a,$ so $a\cap E_i=\emptyset.$ However, $b\notin P_i$ and therefore $b\subset E_i.$ It follows that $M_a\cap E_i=\emptyset$ while $M_b\cap E_i\neq\emptyset$ and hence $M_a\supset M_b$ showing that $a<_{\M}b.$ If $a\in P_i$ and $b\notin P_j$ for any $j\in J,$ then the same argument shows that $a<_{\M}b.$ It means that $<_{\M}$ is proper on atoms of any finite subalgebra and therefore $<_{\M}\in\no_{\B_J}(\C_{\B_J}).$

It remains to show that if $\M$ is a maximal chain that does not extend $\{E_j:j\in J\},$ then $<_{\M}$ is not natural on some finite substructure $A$ of $C_{\B_J}.$  If $S$ is a closed subset of $\E,$ let $\F_S$ denote the filter on $\C_{\B_J}$ determining $S.$ Since $\M$ does not extend $\{E_j:j\in J\},$ there exists $M\in\M$ and $i\in J$ such that $E_i\not\subset M\not\subset E_i.$ Let $x\in M\setminus E_i, y\in E_i\setminus M.$  Since $x\notin E_i,$ there is $a\in \F_x$ that is not compatible with $\F_{E_i}.$ In order for that to happen, it must hold that $a\in P_i.$ Thinking of $a$ as a clopen subset of $\E,$ we have that $a\cap M\neq\emptyset.$ Similarly, since $y\notin M,$ there is $b\in \F_y$ not compatible with $\F_M.$  However, $b$ is compatible with $\F_{E_i}$ and therefore $b\notin P_i.$ Thinking of $b$ as a clopen subset of $\E,$ it follows that $b\cap M=\emptyset.$ Let $a'=a\wedge \neg b\in x.$ Then $a'$ is compatible with $\F_M,$ $a'\in P_i$ and $a'\wedge b=0.$ Recall that $M_a'=\bigcap\{N\in \M:N\cap a\neq \emptyset\}$ and $M_b=\bigcap \{N\in \M:N\cap b\neq \emptyset\}.$ As $M\cap a'\neq\emptyset=M\cap b,$ $M_a'\subset M_b,$  so $b<_{\M} a'.$ It contradicts that $<_{\M}$ is proper on the algebra generated by $a',b,\neg(a'\vee b)$, which implies that $<_{\M}$ is not a normal ordering on $\C_{\B_J}.$ 
\end{proof}

If $\B=\B_u,$ then the generic ideal $P_0^{\C_{\B}}$ is a prime filter and therefore the dual filter $\F_0$ is an ultrafilter. Indeed, if $a\in \C_{\B}\setminus \F_0,$ then the  $4$-element algebra $\{0,1,a\neg a\}\in \B,$ so only one atom is not in $P_0^{\C_{\B}},$ namely $\neg a.$ Hence $\neg  a\in \F_0.$ It means that the closed subset $E_0$ of $\E$ determined by $\F_0$ is exactly one point.

Since $\E$ is a homogeneous space, i.e. for every $x,y\in \E$ there exists a homeomorphism $h$ of $\E$ such that $h(x)=y,$ all groups of automorphisms of $\E$ fixing a single point are topologically isomorphic.

\begin{thm}
Let $G$ be a group of homeomorphisms of the Cantor set $\E$ fixing a point $x.$ Then the universal minimal flow of $G$ is the space of maximal chains of closed subsets of $\E$ containing $\{x\}.$
\end{thm}

\begin{proof}
The proof goes along the same lines as the proof of Theorem \ref{T4}.

Let $\B=\B^u$ and let $x$ be the point given by the ultrafilter $\F_x$ dual to the generic ideal $P^{\C_{\B}}_0.$ As we noted above, it is enough to prove the theorem for this choice of $x$ and by homogeneity of $\E$ it follows for all $x\in \E.$

Let again $\phi:\Phi(\E)\to \no_{\B}(C_{\B})$ be the isomorphism as in Theorem \ref{T3}. Suppose first that $\M$ is a maximal chain of closed subsets of $\E$ extending $\{x\}.$ Notice that $\bigcap \M=\{x\}$ by maximality of $\M.$ Let $A$ be a non-trivial finite substructure of $\C_{\B}$ and let $a,b$ be atoms of $A$ such that $a\in P_0^{\C_{\B}}$ and $b\notin P_0^{\C_{\B}}.$ Then $b\in \F_x$ and therefore $M_b=\bigcap\{M\in \M:b\cap M\neq\emptyset\}=\{x\}.$ Since $a\notin \F_x,$ $M_a$ must be larger than $M_b$ and therefore $a<_{\M} b,$ which shows that $<_{\M}$ on atoms of $A$ is proper.

Conversely, let $\M\in\Phi(X)$ not extend $\{x\}.$ Then $\bigcap \M=\{y\}\in\M$ for some $y\neq x.$ Let $a\in \F_x$ such that $\neg a\in \F_y.$  Then $a\notin P_0^{\C_{\B}}$ while $\neg a\in P_0^{\C_{\B}}$ and $M_a\subset M_{\neg a}.$ That implies $\neg a<_{\M}a$ and therefore $<_{\M}$ is not proper on atoms of the algebra $\{0,1,a,\neg a\}.$ 
\end{proof}

Let $\B=\B_J^u$ and let $\{P_j^{C_{\B}}:j\in J\}$ be the generic ideals on $C_{\B}.$ For each $j\in J,$ let $\F_j$ be the filter dual to $P_j^{C_{\B}}$ and $E_j$ the closed subset of $\E$ determined by $\F_j.$ Then $\{E_j:j\in J\}$ is a decreasing chain of closed sets with an intersection a single point. Indeed, $\bigcap_{j\in J} E_j=\{e\in \E:\F=\bigcup_{j\in J} \F_j\subset e\}.$ Let $a\in \C_{\B}.$ Then the Boolean algebra $\{0,1,a,\neg a\}\in \B.$ Therefore by the definition of $\B_J^u,$ exactly one of  $a$ and $\neg a$ is in  $P_{j_0}^{C_{\B}}$ for some $j_0\in J,$ say $a.$ Then $\neg a\in \F_{j_0},$ which shows that $\F$ is an ultrafilter.   

\begin{thm}
Let $\{E_j:j\in J\}$ be a decreasing chain of closed subsets of $\E$ intersecting in a singleton as above. Then the group of homeomorphisms of the Cantor set fixing $E_j$ for every $j\in J$ is the space of chains of maximal closed subsets of $\E$ containing $\{E_j:j\in J\}.$
\end{thm}

\begin{proof}
The proof is identical to the proof of Theorem \ref{T4}.
\end{proof}

\subsection*{Quotients of power set algebras $\P(\k)/[\k]^{<\l}$}

Throughout this section, let $\l\leq\k$ be two infinite cardinals. We denote by $\P(\k)/[\k]^{<\l}$ the quotient of the power set algebra on $\k$ by the ideal of sets of cardinality less than $\l.$ 

If $\k=\l,$ then $\P(\k)/[\k]^{<\l}$ is homogeneous. If $\l<\k,$ then for every $\mu\in [\l,\k]$ we let $P_{\mu}$ be the ideal on $\P(\k)/[\k]^{<\l}$ consisting of equivalence classes of subsets of $\k$ of cardinality $\leq\mu.$ Then the extended structure
$$\P^{\k}_{\l}=(\P(\k)/[\k]^{<\l}, P_{\mu}:\mu\in[\l,\k])$$
becomes a homogeneous structure. To see that let $A,B$ be finite substructures of $\P^{\k}_{\l}$ and $\phi:A\to B$ an isomorphism. Let $A_0, A_1,\ldots A_n$ be mutually disjoint subsets of $\k$ such that $\bigcup_{i=0}^n A_i=\k$ and $[A_0], [A_1], \ldots, [A_n]$ are atoms of $A.$  Similarly, pick a disjoint partition $B_0, B_1,\ldots, B_n$ of $\k$ consisting of representants of atoms of $B.$ Without loss of generality, we can assume that $\phi([A_i])=[B_i]$ for every $i.$
Then for every $i,$ the sets $A_i$ and $B_i$ have the same cardinality, so we can pick a bijection $b_i:A_i\to B_i.$ Let $b:\kappa\to \kappa$ be a bijection extending every $b_i$ and consider the induced isomorphism $b_{\sim}:\P^{\k}_{\l}\to \P^{\k}_{\l}.$ Then $b_{\sim}|A=\phi$ and $b_{\sim}$ is a total isomorphism, which shows that $\P^{\k}_{\l}$ is homogeneous.

If $\t$ is the order type of $[\l,\k+1],$ then $\B_{\t}$ is the class of all finite substructures of $\P^{\k}_{\l}$ up to an isomorphism.

Having shown that $\na(\B_{\t})$ has all the desired properties, we are ready to apply Theorem \ref{T} from \cite{DB} to compute the universal minimal flow of the group of automorphisms of $\P^{\k}_{\l}.$ 

\begin{thm}
Let $\l\leq \k$ be two infinite cardinals, then the universal minimal flow of the group of automorphisms of $\P^{\k}_{\l}$ is the space $\no(\P^{\k}_{\l}).$
\end{thm}

If there is no isomorphism between $\P(\omega)/\fin$ and $\P(\omega_1)/\fin,$ then the groups of automorphisms of $\P(\omega_1)/\fin$ and $\P^{\omega_1}_{\omega}$ are topologically isomorphic. So we obtain the following corollary.

\begin{cor}
Let $G$ be the group of automorphisms of $\P(\omega_1)/\fin.$ If there is no isomorphism between $\P(\omega)/\fin$ and $\P(\omega_1)/\fin,$ then the universal minimal flow of $G$ is the space $\no(\P(\omega_1)/\fin).$
\end{cor}

If $\P(\omega)/\fin$ and $\P(\omega_1)/\fin,$ then the universal minimal flows of their groups of automorphisms are isomorphic. The universal minimal flow of the group of homeomorphisms of $\omega^*$ (the Stone space of $\P(\omega)/\fin$) was first computed by Glasner and Gutman in \cite{GG}. We reproved the result in the language of Boolean algebras in \cite{DB}.

\begin{thm}\cite{GG}
If $\P(\omega)/\fin$ and $\P(\omega_1)/\fin$ are isomorphic, then the universal minimal flow of the group of automorphisms of $\P(\omega_1)/\fin$ is the space $\no(\P(\omega)/\fin).$
\end{thm}

Following a paper by van Douwen \cite{vD}, we introduce two dense subgroups of the group of automorphisms of $\P^{\k}_{\l}:$ $S_{\k,\l}^*$ and $T_{\k,\l}^*.$

Denote by $T_{\kappa,\l}$ the set of all bijections between subsets $A,B\subset\kappa$ with $|\kappa\setminus A|, |\kappa\setminus B|<\l.$ With the operation of composition, $T_{\kappa,\l}$ is a monoid, but not a group. We can however assign to each $f\in T_{\kappa,\l}$ an automorphism $f^*$ of $\mathcal{P}_{\l}^\k,$ $f^*([X])=[f[X]],$ mapping $T_{\kappa,\l}$ onto a subgroup $T_{\kappa,\l}^*=\{f^*:f\in T_{\kappa,\l}\}$ of the group of automorphisms of $\mathcal{P}^{\k}_{\l}.$ Since the automorphisms in $T_{\kappa,\l}^*$ are induced by a pointwise bijection between subsets of $\kappa,$ we call them \emph{trivial}. Inside of $T_{\kappa,\l}^*$ we have a normal subgroup of those automorphisms induced by a true permutation of $\kappa,$ let us denote it by  $S_{\kappa,\l}^*.$ Shelah \cite{S1} (see also \cite{SS}) proved that consistently every automorphism of $\mathcal{P}(\omega)/{\rm fin}$ is trivial. This has been extended to $\mathcal{P}(\kappa)/{\rm fin}$ for all cardinals $\kappa$ in \cite{V}. Of course, consistently the two groups are different. 

By density of $S_{\kappa,\l}^*$ and $T_{\kappa,\l}^*$ in the group of automorphisms of $\P^{\k}_{\l},$ we can apply Theorem \ref{T} to prove the following.

\begin{cor}
The universal minimal flow of the groups $S_{\k,\l}^*$ and $T_{\k,\l}^*$ is the space $\no(\P^{\k}_{\l}).$
\end{cor}

\textbf{Acknowledgement.} I would like to thank Stevo Todor\v{c}evi\'c for his kind support and advice while pursuing this work.

\bibliographystyle{alpha}
\bibliography{Linear}

\end{document}